\documentclass[12pt]{article}
\usepackage{amsmath,amssymb,amsfonts,amsthm}
\usepackage{eucal}

\newcommand{\C}{\mathbb{C}}

\newcommand{\cO}{\mathcal{O}}
\newcommand{\cV}{\mathcal{V}}

\newcommand{\cL}{\mathcal{L}}

\newcommand{\Pp}{\mathbb{P}^1}
\newcommand{\Pt}{\mathbb{P}^2}

\newcommand{\lang}{\left\langle}
\newcommand{\rang}{\right\rangle}

\newcommand{\cH}{\mathcal{H}}

\newcommand{\zee}{\mathfrak{z}}
\newcommand{\cM}{\mathcal{M}}

\DeclareMathOperator{\End}{End}
\DeclareMathOperator{\Hilb}{Hilb}
\DeclareMathOperator{\Ext}{Ext}

\DeclareMathOperator{\ch}{ch}
\DeclareMathOperator{\td}{td}

\DeclareMathOperator{\Gr}{Gr}

\DeclareMathOperator{\Tr}{Tr}

\DeclareMathOperator{\Iso}{Iso}
\newtheorem*{thma}{Theorem A}
\newtheorem{thm}{Theorem}

\begin{document}
\title{$K$-theory of moduli spaces of sheaves and large Grassmannians}
\author{Erik Carlsson}
\maketitle
\begin{abstract}
We prove a theorem classifying the equivariant $K$-theoretic pushforwards of the product of arbitrary Schur functors applied to the tautological
bundle on the moduli space of framed rank $r$ torsion-free sheaves on $\mathbb{P}^2$, and its dual. 
This is done by deriving a formula for similar coefficients on 
Grassmannian varieties, and by thinking of the moduli space as a class in the $K$-theory of the Grassmannian, in analogy with the
construction of the Hilbert scheme when the rank is one.
Our motivations stem from some vertex operator calculus studied recently by
Nekrasov, Okounkov, and the author when the rank is one, with applications to four-dimensional gauge theory.
\end{abstract}
\section{Introduction}
Let $\cM_{r,n}$ denote the moduli space of framed, rank-$r$ torsion-free sheaves on $\Pt$, together with a framing at the line at infinity,
with second Chern class $c_2=n$ described in detail in \cite{HL,Nak2}.
There is a second description of $\cM_{r,n}$ as a symplectic quotient of the space of ADHM data by the general linear group $GL_n$.
This moduli space is a smooth variety of complex dimension $2rn$, with a
tautological $n$-dimensional complex vector bundle whose fiber is given by
\begin{equation}
\label{cV}
\cV\big|_F = H^1_{\Pt}(F(-\Pp_\infty)),\quad F \in \cM_{r,n},
\end{equation}
where $\Pp_\infty$ is the line at infinity on $\Pt$.
This bundle can be defined either using the tautological sheaves on the moduli space, or by taking the fiber product with the standard
representation of $GL_n$ in the ADHM description. When the rank $r$ is one, $\cM_{r,n}$ is canonically isomorphic to the Hilbert scheme 
of points on $\C^2$ by identifying the subscheme with its ideal sheaf in projective space. In this case, $\cV$ is the bundle whose fiber over
a subscheme is the space of global sections of its ring of functions. Its determinant
\[\cL = \det(\cV)\]
is an important line bundle, which will be used later.

There is a well-known torus action of $G = T^2\times T^r$ on $\cM_{r,n}$, whose fixed-locus is a finite set,
which is proved by Nakajima in \cite{Nak4}.
The two-dimensional torus acts by pulling back sheaves via the action
\begin{equation}
\label{T2}
T^2 \circlearrowright \Pt,\quad (z_1,z_2)\cdot (x_1,x_2,x_3) = (z_1^{-1}x_1,z_2^{-1}x_2,x_3),
\end{equation}
preserving the line at infinity $x_3=0$. The $r$-dimensional torus, whose coordinates are labeled by $w_j$,
acts by composition on the framing, as the standard torus inside $GL_r$.
The moduli space is noncompact, but there is still a well-defined pushforward in equivariant $K$-theory,
\begin{equation}
\label{exp}
\lang E \rang_X = \sum_{i} (-1)^i \ch H^i_X(E),\quad E \in K_G(X).
\end{equation}
The moduli space is noncompact, but Nakajima proved in \cite{Nak4} that the sheaf cohomology groups of an equivariant bundle on $\cM_{r,n}$ have finite-dimensional $T$-eigenspaces, making the above expectations well defined. He 
showed furthermore that the localization formula combined with Riemann-Roch applies, and it calculates
its expectation as a rational function of $z_i,w_j$.
This is given in equation \eqref{modexpdef}, which could be taken as its definition.

For each $r,n$, we define an inner product on $\Lambda$, the space of symmetric functions in infinitely many variables, by
\begin{equation}
\label{modexp}
(s_\mu,s_\nu)_{\cM_{r,n},k} =\lang \mathbb{S}_\mu(\cV)^* \mathbb{S}_\nu(\cV) \cL^k \rang_{\cM_{r,n}}.
\end{equation}
Here $s_\mu$ are the Schur polynomials, and $\mathbb{S}_\mu$ are the Schur functors. 
If both $\mu,\nu$ are the empty partition, and $k=0$, this is the massless $K$-theoretic generalization of one of Nekrasov's partition functions $Z_n$ studied in \cite{Nak5}. A central application we have in mind is to a more general partition function of Nekrasov's which has a mass term, and to some recent results involving vertex operators acting on the equivariant $K$-theory groups of the Hilbert scheme studied by Nekrasov, Okounkov, and the author \cite{CNO}. See also \cite{IKS} for more on the $K$-theoretic partition function from the point of view of string theory.

Our main theorem classifies these coefficients:
\begin{thma}
Fix $f,g \in \Lambda$, and $r,n \geq 0$.
Then $(f,g)_{\cM_{r,n},k}$ is the unique element of $\C(z_i,w_j)$ satisfying
\begin{itemize}
\item[a.] There exists an element $F \in \C(z_i,w_j)[x_i,y_j]$ such that
\[(f,g)_{\cM_{r,n},k} = F(z_i^k,w_j^k)\]
for all $k$.
\item[b.] There exists $k_0$ such that for any $k \geq k_0$
\[(f,g)_{\cM_{r,n},k} = 
\Tr_\Lambda \varphi_{(1-M)p_1} m_f^* \Gamma_+(W)\pi_n m_{e_n^k} \Gamma_-(z_1z_2\overline{W}) m_g, \]
where $M = (1-z_1)(1-z_2)$ and $W = w_1 + \cdots + w_r$.
\end{itemize}
\end{thma}
\noindent
Here $\varphi_{(1-M)p_1}$ is a homomorphism acting diagonally in the power-sum basis $p_\mu$,
the $\Gamma_{\pm}$ are vertex operators acting on $\Lambda$ by a homomorphism and a multiplication operator respectively, $\pi_n$ is a projection map
onto the subspace of finitely many variables, and $m_f$ and $m_f^*$ are the multiplication map by $f$ and its dual under the Hall inner product. These are defined in section \ref{prelim}.

The first condition and the uniqueness statement are obvious from definitions, but are important. 
The intended application of the theorem would be to extend formulas that exist when
the rank is one, including many that are based on Haiman's theory \cite{H1,H2}, to higher rank. To establish a formula for this inner product for all $k$, especially the case when it equals zero, one only needs to show that it satisfies the two conditions in the theorem, which amounts to studying the large $k$ behavior. Furthermore, it is a necessary constraint that $k$ be large. When the rank is zero, the left hand side is automatically zero, but the vertex operator expression only vanishes for large $k$, for degree reasons.
In fact, the right-hand-side of the theorem is actually always a holomorphic function of $z_i$ at the origin, whereas
the left-hand-side is a meromorphic function with singularities at $z_i=0$, unless $k$ is large enough to cancel them.
\subsection{} The motivations for this theorem start with some remarkable properties of certain integrals over $\Hilb_n \C^2$,
and sometimes $\cM_{r,n}$, which are of mathematical and physical interest. A central example is the aforementioned partition function of Nekrasov with a mass parameter $m$,
defined mathematically as a generating function
\[Z = \sum_n q^n \int_{\cM_{r,n}} e(T_m),\]
where $T_m$ is the tangent bundle, but endowed with an additional action of a one-dimensional torus, with Lie parameter $m \in \mathfrak{t}$.
When the rank is one, and the group action \eqref{T2} is replaced by the simpler action
\begin{equation}
\label{CYaction}
z \cdot (x_1,x_2,x_3) = (zx_1,z^{-1}x_2,x_3),
\end{equation}
the partition function turns out to be a power of the eta function times a power of the generating variable $q$ \cite{NO}. 
See \cite{N,O} for a readable introduction to this topic, and the gauge theoretic background behind the partition function.
Underlying these properties are actions of structures that
related to conformal field theory (Heisenberg, Ka\c{c}-Moody algebras, vertex operators, the Virasoro algebra, double affine Hecke algebras, and other generalizations),
on the cohomology groups of the Hilbert scheme. This is a quite large subject, see \cite{Ba,CNO,CO,Groj,L2,LQW,Lic,MO,Nak1,Nak6,OP,SV},
to name a few.

Theorem A stemmed from an attempt to generalize
a vertex operator that has been studied recently by Okounkov, Nekrasov, and the author. 
In \cite{NO}, Okounkov and Nekrasov used a well-known vertex operator in their study of the partition function $Z$ and Seiberg-Witten theory.
In \cite{CO}, Okounkov and the author gave it a geometric definition, and
extended to an operator
\[W(\cL) : \cH \rightarrow \cH,\quad \cH = \bigoplus_n H^*(\Hilb_n S)\]
on the cohomology groups the Hilbert scheme of a general surface $S$, with a line bundle $\cL$.
When $\cL$ is the trivial bundle endowed with a nontrivial torus action, the weight $m$ of the action plays the role of the mass parameter mentioned above. We defined this operator as a characteristic class of a certain vector bundle 
on a pair of Hilbert schemes built from $\Ext$-groups of the canonical sheaves. We then proved a formula in terms of Nakajima's Heisenberg operators
mentioned above, which is the sense in which $W$ is a vertex operator. Most recently, Nekrasov, Okounkov and the author defined a $K$-theoretic version of this operator, which we call $\tilde{W}$, which limits to $W$ to low order in the equivariant parameters, by taking Chern characters \cite{CNO}.
The main theorem is that it equals a certain vertex operator under the Bridgeland, King, and Reid isomorphism, generalizing the previous formula.
The essential machinery was Haiman's character theory of this isomorphism, combined with some insight about how to decompose $\tilde{W}$, and the theory of the Fourier transform in symmetric functions.

In a typical application of this operator, the trace of $W$ against another well-chosen operator $c$, transparently calculates a quantity of interest.
But because of the description of $W$ as a vertex operator, the resulting formulas are often tractable and intriguing,
and are sometimes characterized by properties such as modularity.
For instance, the trace of $W$ against the operator of multiplication by $q^n$ on the $\Hilb_n \C^2$ component of $\cH$ is the
partition function with a mass term, which turns out to be a modular form up to a power of $q$ if the Calabi-Yau action \eqref{CYaction} is used \cite{NO}.
When $c$ is the operator of the cup-product by characteristic classes of $\cV$, one obtains formulas for the analogous correlation functions as coefficients
of theta functions, which are quasi-modular forms \cite{C2,KZ}.
When the surface varies, other generalizations of modularity appear. Using the full $W(\cL)$ on a general surface and letting $c$ equal a Calogero operator of Manfred Lehn's
produces the Eisenstein series $E_3$ \cite{CO}.
When $c$ is a certain charge operator coming from various other affine Lie algebras, one can extract properties of the dual partition function in higher rank,
as Okounkov and Nekrasov did in \cite{NO}.
In some cases, vector-valued modularity has appeared \cite{Ct}.
A common aspect of these calculations and the other operators mentioned in the
first paragraph is that on obtains a function of $q$ with intriguing properties,
and one obtains new cohomological formulas by extracting its Taylor coefficients.

The $K$-theoretic derivation of the vertex operator in \cite{CNO} is the most meaningful and general, but it is not yet clear how to extend it to higher rank.
As we mentioned, an essential tool is Haiman's character theory of the BKR isomorphism, which is a pushforward and a pullback over the
singular fiber product
\[\Iso_n = \Hilb_n \C^2 \times_{\C^{2n}/S_n} \C^{2n},\]
called the isospectral Hilbert scheme, see \cite{H1,H2}.
The map from $\Hilb_n \C^2$ to $\C^{2n}/S_n$ is the Hilbert-Chow map, which only makes sense in rank $1$.
Since Nekrasov's partition function has to do with the ADHM or commuting operator description of $\cM_{r,n}$ \cite{N},
it is desirable to produce an operator which is not based on the BKR isomorphism or
any other sophisticated black box from algebraic geometry, even in rank one.

In this paper, we abandon the search for an operator on $K_T$ which is a vertex operator under the BKR isomorphism, and instead study the
inner product \eqref{modexp}. By the Cauchy expansion \eqref{cauchy}, which is a part of the general reasoning of \cite{CNO}, this inner product also recovers the $K$-theoretic version of the partition function with the mass parameter,
\[\Tr A \Gamma_n u^d = Z_n(z_1,z_2,u),\]
where $Z_n(z_1,z_2,u)$ is the $K$-theoretic generalization of the partition function as in \cite{Nak5}, but with the aforementioned mass parameter $u = e^m$, $d$ is the degree operator,
$\Gamma_n$ is the operator corresponding to \eqref{modexp},
\[(\Gamma_n f,g) = (f,g)_{\cM_{r,n},0},\]
and $A$ is a simple expression in some of the vertex operators described in section \ref{prelim}.
This function limits to the original cohomological partition function $Z$ by
\[\sum_n q^n \lim_{t \rightarrow 0} Z_n(e^{t\epsilon_1},e^{t\epsilon_2},e^{tm}) = Z.\]

The proof of theorem A is completely different from the ideas in previous papers.
We first assume that the equivariant parameters $z_i$ take fixed values with norm less than $1$, and introduce an auxiliary meromorphic
function of one complex variable,
\[F(v) = \lim_{N \rightarrow \infty} F_N(v).\]
When $k$ is large enough, $F$ converges in a disc of radius $1+\epsilon$ about the origin, and its value at $v=1$ is the quantity of interest,
\[F(1) = (f,g)_{\cM_{r,n},k}.\]
The Taylor coefficients of $F_N$ about $v=0$ are given by certain expectations similar to \eqref{modexp} but with $\cM_{r,n}$ replaced by a certain
Grassmannian variety of size depending on $N$.
In section \ref{grasssection}, we prove a theorem that determines coefficients in terms of vertex operators when $N$ is large.
We may switch limits and obtain an exact answer for $F(v)$ when the norm of $v$ is small, but not when it equals one,
because $F_N$ has poles inside the unit disc. However, the answer turns out to be a meromorphic function
which equals the right-hand-side of theorem A
when $v=1$, and agrees with $F(v)$ near the origin. The theorem then follows by analytic continuation.

When the rank is one, there is a reasonable explanation for these constructions.
The Hilbert scheme of projective space is constructed as a subvariety of a large Grassmannian variety of subspaces of the graded ring
in a high degree.
In a similar way, the Hilbert scheme of $\C^2$ sits in the Grassmannian of subspaces of $R = \C[x,y]$ cut off in high degree, determined by $N$, and the vector bundle $\cV$ on $\Hilb_n \C^2$ is just the pullback of $V'$, the complement of the tautological bundle.
One can then think of $F(v)$ a approximating the expectation of Schur functors applied to $V'$ and its dual
against something close to a fundamental class of the Hilbert scheme, by the projection formula.
This is a reasonably down-to-earth idea, but it is unusual to make practical use of the construction of the moduli space in a calculation.
More often, the construction establishes the existence, but computations rely only on the moduli functor and deformation theory.

\emph{Acknowledgements.} The author would like to thank Ezra Getzler, Nikita Nekrasov, and Andrei Okounkov for many valuable discussions. Getzler first introduced the author to Shaun Martin's theorem.

\section{Preliminaries}
\label{prelim}
We need some preliminaries about symmetric functions. Everything in this section is standard, so we omit the details and proofs. We refer to \cite{Mac} for the background and notation on symmetric functions, to \cite{BO,FBZ,K,PS} for the infinite-wedge representation and the boson-fermion correspondence, and to \cite{H2} for the plethystic notation.

Let
\[\Lambda = \C[p_1,p_2,...],\quad p_k = \sum_i p_i^k\]
be the ring of symmetric polynomials in infinitely many variables. There is a projection map
\[\pi_n : \Lambda \rightarrow \Lambda,\quad \pi_n(s_\mu) = \left\{\begin{array}{cc} s_\mu & \mbox{$\ell(\mu) \leq n$} \\ 0 & \mbox{otherwise}\end{array}\right.\]
onto the symmetric functions of finitely many variables under the Hall inner product
\begin{equation}
\label{hall}
(p_\mu,p_\nu) = \delta_{\mu,\nu} \zee(\mu),
\end{equation}
in which the Schur polynomials $s_\mu$ are orthonormal. Here $\ell(\mu)$ denotes the length of the partition $\mu$.
The inner product
\[ (f,\pi_n g) = [x_i^0] f(x_i)g(x_i^{-1})\Delta_n \overline{\Delta}_n,\quad \Delta_n = \prod_{1 \leq i<j\leq n} (x_i-x_j)\]
agrees with the Hall inner product in finitely many variables, denoted in Macdonald's book by $(f,g)_n$.

Any homomorphism on $\Lambda$ may be prescribed by evaluating the generators $p_k$ at a value.
If $f$ is a function of some indeterminants, let $\varphi_f$ be the homomorphism
\[\varphi_f : p_k \mapsto f_k,\]
where $f_k$ is the plethystic evaluation of $f$ at $x=x^k$ for every variable $x$ appearing in $f$.
This definition implicitly depends on a fixed set of chosen variables,
but for our purposes we can simply include all the variables in the paper.
In this notation, the homomorphism of evaluation at finitely many variables would be described by
\[\varphi_{x_1+\cdots + x_n} g = g(x_1,...,x_n).\]
$f$ may even contain the variables $p_i$, in which case
$f_k$ denotes the evaluation at $p_i = p_{ik}$, which is compatible with the above evaluation.
Another example is
\[\varphi_{-p_1} f = (-1)^{\deg(f)} \omega f,\]
where $\omega$ is the Macdonald involution. 
Let us also define a conjugation by
\[\overline{f} = f\big|_{x=x^{-1}}\]
for each variable $x$.
In this paper we will often denote a torus representation and its character by the same letter, so that
\[ \varphi_{X} s_\mu = \varphi_{\ch X} s_\mu =\ch \mathbb{S}_\mu(X).\]
See Haiman \cite{H2} for more on this \emph{plethystic} notation, which is useful in describing the Frobenius
characters of $S_n$-equivariant modules over rings.

Next we recall the \emph{infinite-wedge representation}. Let $\Lambda^{\infty/2} = \bigoplus_c \Lambda^{\infty/2}_c$, where
\[\Lambda_c^{\infty/2} = \bigoplus_{\mu}\C\cdot v_{\mu,c},\quad v_{\mu,c} = v_{\mu_1+c}\wedge v_{\mu_2-1+c} \wedge v_{\mu_3-2+c} \wedge \cdots.\]
It defines a projective representation 
\[\rho : GL_\infty \rightarrow GL(\Lambda^{\infty/2}_c)\]
of the infinite general linear group on the vector space with basis $v_i$, indexed by the integers.
There are many possible definitions of $GL_\infty$
with different convergence conditions on the matrix elements. In this paper, we are content with matrices such that $x_{ij}$
vanishes whenever $i-j$ exceeds some number.

The action is the easiest to describe on the Lie algebra level. Define the usual wedging and contracting operators by
\[\psi_i : \Lambda^{\infty/2}_c \rightarrow \Lambda^{\infty/2}_{c+1},\quad\psi_i(v_{\mu,c}) = v_i \wedge v_{\mu,c},\]
and let $\psi^*_i$ be the dual operator under the inner product on $\Lambda^{\infty/2}_c$ in which $v_{\mu,c}$ are orthonormal.
If $X_{ij}$ is an infinite matrix in the Lie algebra such that $X_{ij} = 0$ whenever $i-j$ exceeds some number, define the action by
\[\rho'(X) = \sum_{i,j} X_{ij} :\psi_i\psi_j^*:\]
where
\[ :\psi_i \psi_j^*: = \left\{\begin{array}{rl} \psi_i \psi_j^* & \mbox{$i \neq j$ or $i=j>0$} \\ -\psi_j^*\psi_i & \mbox{otherwise.}\end{array}\right.\]
On the group level, the matrix elements are
\begin{equation}
\label{rho}
\rho(x)_{\mu,\nu} = \lim_{N \rightarrow \infty}
\frac{\det\left(x_{\mu_i-i+c,\nu_j-j+c}\right)_{0\leq i,j \leq N}}{\det\left(x_{-i+c,-j+c}\right)_{0\leq i,j\leq N}}.
\end{equation}
See Ka\c{c} \cite{K} for more on this representation on the Lie algebra side, and Pressley and Segal \cite{PS} for the group side.

An infinite matrix of the form $x_{ij} = x_{j-i}$ induces an operator that commutes with the translation isomorphism
\[Q : \Lambda^{\infty/2}_c \rightarrow \Lambda^{\infty/2}_{c+1},\quad Qv_{\mu,c} = v_{\mu,c+1}.\]
For instance, the usual Heisenberg operators on this space are just the matrices with infinite off-diagonal strips,
\[\alpha_n = \sum_i \psi_{i} \psi^*_{i+n},\quad n \neq 0\]
and $\alpha_0$ is the operator of multiplication by $c$ on $\Lambda^{\infty/2}_c$.
It is easy to see that they satisfy the commutation relations
\begin{equation}
[\alpha_i,\alpha_j] = i\delta_{ij}.
\end{equation}
On the group level, we have the following vertex operators,
\begin{equation}
\label{gpm}
\Gamma_{\pm}(f) = \rho(g_{\pm}(f)),\quad g_{\pm}(f)_{ij} = \varphi_{f}h_{\pm(i-j)}
\end{equation}
where $h_k$ are the complete elementary symmetric polynomials for $k \geq 0$ and are zero for $k<0$.
$\Gamma_-(f)$ is not quite an operator for any particular values of the variables in $f$ because we defined $\Lambda^{\infty/2}_c$
as a direct sum, and instead should be regarded as an operator-valued distribution on the variables, when this is possible, see \cite{FBZ}. 
Using the exponential map,
\begin{equation}
\label{exp}
\Gamma_\pm(f) = \exp\left(\sum_{k\geq1} \frac{f_k}{k}\alpha_{\pm k}\right).
\end{equation}

The \emph{boson-fermion correspondence} is an isomorphism between $\Lambda_c^{\infty/2}$
and the irreducible lowest-weight representation of the Hiesenberg algebra, determined
by expressing the fermionic operators $\psi_i,\psi_i^*$ in terms of $\Gamma_\pm$. It says
\begin{equation}
\label{bf}
\psi_i = \tilde{\psi}_{i-c}Q,\quad \psi_i^* = Q^{-1}\tilde{\psi}^*_{i-c},
\end{equation}
where
\begin{equation}
\label{bf2}
\tilde{\psi}_i = [z^i] \Gamma_-(z)\Gamma_+^{-1}(z^{-1}),\quad\tilde{\psi}^*_i = [z^{-i}] \Gamma^{-1}_-(z)\Gamma_+(z^{-1}),
\end{equation}
and the coefficients of $z$ are taken as elements of $\End(\Lambda^{\infty/2}_c)((z))$. See Ka\c{c} \cite{K} for
more details, or Frenkel and Ben-Zvi \cite{FBZ} for a stronger version of this statement in the context of vertex algebras.

Consider the isomorphism of inner product spaces
\[\Phi_c : \Lambda \rightarrow \Lambda^{\infty/2}_c,\quad s_\mu \mapsto v_{\mu,c},\]
which is also often referred to as the boson-fermion correspondence.
It comes from multiplying $s_\mu$ by the antisymmetric polynomial
\[\Delta_m = \prod_{1 \leq i<j\leq m} (x_i-x_j),\]
and normalizing by the appropriate power of $x_1\cdots x_m$
depending on $c$, as the number of variables approaches infinity.
Under this map, the Heisenberg operators map to
\[\Phi^{-1}_c\alpha_n\Phi_c f= \left\{\begin{array}{rl} p_{-n} f & n<0 \\ n \partial_{p_n} f& n>0 \end{array}\right.\]
which is of course consistent with the above Heisenberg relations. The vertex operators are even more interesting. By the above identity and \eqref{exp},
\begin{equation}
\label{PhicGamma}
\Phi_c^{-1}\Gamma_+(f) \Phi_c = \varphi_{p_1+f},\quad 
\Phi_c^{-1}\Gamma_-(f)\Phi_c g = \Omega(p_1f)g,
\end{equation}
where
\begin{equation}
\label{Omega}
\Omega(f) = \exp\left(\sum_{k\geq 1} \frac{f_k}{k}\right) = \sum_\mu \frac{\varphi_{f} p_\mu}{\zee(\mu)},
\end{equation}
assuming this series converges. The resulting rational function also makes sense if $f$ is a Laurent polynomial with integer coefficients,
even if the above summation does not actually converge for any particular value of the variables.
Notice that
\[\Omega(f+g) = \Omega(f)\Omega(g).\]
It follows immediately from \eqref{PhicGamma} that $\Gamma_\pm$ satisfy the commutation relations
\begin{equation}
\label{GammaCR}
\left[\Gamma_\pm(f),\Gamma_\pm(g)\right] =0,\quad \Gamma_+(f) \Gamma_-(g) = \Omega(fg) \Gamma_-(g) \Gamma_+(f).
\end{equation}
In this paper, we will often denote operators and their conjugate by $\Phi_c$ by the same name.

The projection operator $\pi_n$ also has an interesting expression under $\Phi_c$. The smallest index $i$ such that $v_i$ is not a term in $v_{\mu,c}$ is $-\ell(\mu)+c+1$. 
Macdonald's involution $\omega$ swaps $s_\mu$ with $s_{\mu'}$, where $\mu'$ is the transpose of $\mu$, and the indices $i$ in $v_{\mu',c}$
are the negatives of the indices missing from $v_{\mu,-c-1}$. It follows that 
\begin{equation}
\label{bfpi}
\Phi_c^{-1} \omega \pi_n \omega \Phi_c v_{\mu,c} = \left\{\begin{array}{rl} v_{\mu,c} & \mbox{if $\mu_1+c \leq n$} \\ 0 & \mbox{otherwise.}\end{array}\right.
\end{equation}
In fact, this operator equals $\rho(D_{n+c})$, where
\[D_{a,ij} = \left\{\begin{array}{rl} 1 & \mbox{if $i=j \leq a$} \\ 0 & \mbox{otherwise.}\end{array}\right.\]
This is obviously not an invertible matrix, but equation \eqref{rho} still makes perfect sense as long as $a \geq c$.

\section{Grassmannians}\label{grasssection} Let $X$ be a complex representation of a torus $T$, let $N$ be its dimension, and
let $\Gr_m=\Gr_m(X)$ denote the Grassmannian of complex $m$-dimensional subspaces of $X$. It has a tautological
$m$-dimensional bundle $V$, and a complementary bundle
\[V' = X/V,\quad \dim(V') = n = N-m,\]
which are equivariant under the induced action of $T$.

Consider the following inner product on $\Lambda$, 
\begin{equation}
\label{G}
(s_\mu,s_\nu)_{\Gr_m} = \langle \mathbb{S}_\mu(V')^* \mathbb{S}_\nu(V) \rangle_{\Gr_m}, 
\end{equation}
where
\[\lang E \rang_Y = \sum_i (-1)^i \ch H^i_{Y}(E) = \int_{Y}\ch( E) \td \]
is the Euler characteristic of an equivariant bundle $E$, and the second equality is the Riemann-Roch formula.
If $X$ has no multiple weights as a $T$-representation, then the fixed points of the Grassmannian are discrete, and \eqref{G} is determined
by the Atiyah-Bott localization theorem,
\begin{equation}
(f,g)_{\Gr_m} = \sum_{V \in \Gr_m^T} f(\overline{V'})g\left(V\right)\Omega(\overline{V'}V).
\label{Grloc}
\end{equation}
where $f(V)$ is shorthand for $\varphi_V f$, we are using $V$ to denote both the vector space and its character as a $T$-representation.

Our first theorem calculates these coefficients for large $m$, with no restrictions on $n$. The motivation for such a formula is that the Hilbert scheme of $n$ points is constructed as a subvariety of a Grassmannian of codimension $n$ subspaces of a large space.
\begin{thm}
\label{grassthm} If $f=s_\mu$ with $\ell(\mu) \leq m$, then
\[(f,g)_{\Gr_m} = \left(f, \Gamma_+(X)\pi_n \Gamma_+^{-1}(X)\varphi_{-p_1} \pi_m g\right). \]
%
%
\end{thm}
\emph{Remark}.
When $g=1$, this formula becomes
\[\lang f(V^*) \rang_{\Gr_m} = (\varphi_{\overline{X}-p_1} f,1)_{\Gr_m} = f(\overline{X}),\]
which was proved by Edidin and Francisco \cite{EF}, using Borel-Weil theory.
\begin{proof}
We may assume that $m\geq \ell(\nu)$, because otherwise both sides obviously vanish. Using 
\[f(\overline{V'}) = \left(\varphi_{-p_1}\Gamma_+(\overline{X}) f\right)(\overline{V}),\]
and Shaun Martin's theorem \cite{Mar} for the Grassmannian as a symplectic quotient of $U_m$,
%
%
\begin{equation}
\label{grassexp}
( f,g)_{\Gr_m} = \left(\varphi_{-p_1} \Gamma_+(\overline{X})f,g\right)_{\lambda,m},
\end{equation}
where
\begin{equation}
\label{detsp}
\left(f,g\right)_{\lambda,m} = \left(\lambda \boxtimes \cdots \boxtimes \lambda \right) \Delta_m \overline{\Delta}_m f(x_1^{-1},...,x_m^{-1})g(x_1,...,x_m),
\end{equation}
and $\lambda : \C(x) \rightarrow \C$ is the linear functional
\[\lambda(x^i) = \lambda_i = \int_{\mathbb{P}(X)} \ch \cO(-i) \td.\]
If $X$ has distinct weights given by $v_i$, then
%
%
\[\lambda(f) = \sum_i f(v_i) \Omega\left(\sum_{j \neq i} v_iv_j^{-1}\right)\]
and in fact equating \eqref{Grloc} and \eqref{detsp} just amounts to grouping terms.

From, for instance, theorem 5.1 from Hartshorne \cite{Ha}, $\lambda = \lambda'+\lambda''$, where
\[\lambda'_i = \left\{\begin{array}{cc} 
h_{-i}(\overline{X}) & i \leq 0 \\
0 & \mbox{otherwise} \end{array}\right.\]
and
\[\lambda''_i = \left\{\begin{array}{cc}(-1)^{N+1}h_{i-N}(X) \det(X) & i \geq N\\
0 & \mbox{otherwise,} \end{array}\right.\]
and $\det$ means $e_m$.
By expanding $\Delta_m\overline{\Delta}_m$, and using $\ell(\mu),\ell(\nu) \leq m$,
\[(s_\mu,s_\nu)_{\lambda,m} = \det(\Lambda_{\mu_i-i+1,\nu_j-j+1})_{1\leq i,j \leq m}\]
where $\Lambda_{ij} = \lambda_{j-i}$. If there were no $\lambda''$ term, the answer would be
\begin{equation}
(s_\mu,s_\nu)_{\lambda',m} = \left(\Gamma_+(\overline{X}) s_\mu,s_\nu\right),
\label{lambdaprime}
\end{equation}
using \eqref{gpm}.
However, $\lambda''$ does contribute to the determinant, and the terms have indices
\[\gamma = \left\{(a,b) \big|a \leq m,\ \mu_a-a+1 +N\leq \nu_b-b+1 \right\}\]
%
%
%
%
which lies in the lower-left corner of the matrix.

Given a matrix $A = (a_{ij})$, one can expand its determinant about an entry $(i,j)$ by
\[\det(A) = \det(A') + (-1)^{i+j}a_{ij}\det(A^{(ij)})\]
where $A^{(ij)}$ is the minor obtained by deleting the row $i$ and column $j$,
and $A'$ is the matrix obtained from $A$ by setting $a_{ij}=0$.
Expanding about each entry in $\gamma$,
\[(s_\mu,s_\nu)_{\lambda,m} = \sum_{S \subset \gamma} \left(\prod_k \lambda''_{j_k-i_k}\right)
\left(\psi^*_{i_1}\cdots \psi^*_{i_c} s_\mu,
\psi^*_{j_1}\cdots \psi^*_{j_c} s_\nu\right)_{\lambda',m} ,\]
where we have used $\psi_k^*$ to delete entries from the matrix with the appropriate sign,
\[S = \{(a_1,b_1),...,(a_c,b_c)\},\quad (i_k,j_k) = (\mu_{a_k}-a_k+1,\ \nu_{b_k}-b_k+1),\]
and we are using $\Lambda$ and $\Lambda^{\infty/2}_{-c}$
interchangeably. The inner product does not depend on $c$ because $\lambda'_{j-i}$ is constant along the diagonals $j-i=k$.
Since $\ell(\mu) \leq m$, we have
$a \leq m$ if and only if $\mu_a-a+1 \geq -m+1$. We may therefore rewrite the answer as
\[\sum_{i_k,j_k} \left(\prod_k \lambda''_{j_k-i_k+N}\right)
\left(\psi^*_{i_1-N}\cdots \psi^*_{i_c-N} s_\mu,
\psi^*_{j_1}\cdots \psi^*_{j_c} s_\nu\right)_{\lambda',m} ,\]
over all lists $(i_1,j_1),...,(i_c,j_c)$ such that
\[n+1 \leq i_1 < \cdots < i_c,\quad j_1 < \cdots < j_c.\]

Inserting the definition of $\lambda'$ and $\lambda''$, and using \eqref{lambdaprime}, 
this expression becomes
\[ \sum_{i_k,j_k} (-1)^{c(N+1)}\det(X)^{c} \left(\prod_k h_{j_k-i_k}(X) \right)\]
\[\left(\Gamma_+(\overline{X}) \psi^*_{i_1-N}\cdots \psi^*_{i_c-N} s_\mu, \psi^*_{j_1}\cdots \psi^*_{j_c} s_\nu\right)=\]
\[\sum_{i_k,j_k} (-1)^{c(N+1)} \det(X)^{c} \left(\prod_k h_{j_k-i_k}(X) \right)\]
\[\left(\Gamma_+(\overline{X}) s_\mu, \left(\sum_a (-1)^a e_a(\overline{X})\psi_{i_c+a-N}\right) \cdots\right.\]
\[\left. \left(\sum_a (-1)^a e_a(\overline{X})\psi_{i_1+a-N}\right) \psi^*_{j_1}\cdots \psi^*_{j_c} s_\nu\right)=\]
\[\sum_{i_k,j_k} (-1)^{c} \left(\prod_k h_{j_k-i_k}(X) \right)\]
\[\left(\Gamma_+(\overline{X}) s_\mu, \left(\sum_a (-1)^a e_a(X)\psi_{i_c-a}\right) \cdots\right.\]
\[\left. \left(\sum_a (-1)^a e_a(X)\psi_{i_1-a}\right) \psi^*_{j_1}\cdots \psi^*_{j_c} s_\nu\right)=\]
%
%
\begin{equation}
\label{psistuff}
\sum_{i_k,j_k} Z_{i_1,j_1} \cdots Z_{i_c,j_c}\left(\Gamma_+(\overline{X}) s_\mu,\psi_{i_c} \cdots\psi_{i_1} \psi^*_{j_1}\cdots \psi^*_{j_c} s_\nu\right),
\end{equation}
%
%
where 
\begin{equation}
\label{Z}
Z_{ij} = -\sum_{a\geq n+1-i}(-1)^{a} e_{a}(X) h_{j-i-a}(X).
\end{equation}

After a little work, one finds that the infinite matrix $Z$ has an interesting shape,
\[Z_{ij} = \left\{\begin{array}{rl} -1 & \mbox{if $i=j \geq n+1$} \\ Z'_{ij} & \mbox{if $i \leq n, j \geq n+1$} \\ 0 & \mbox{otherwise.}\end{array}\right.\]
In other words, it is the sum of a nilpotent matrix $Z'_{ij}$ which squares to zero, lying entirely in one quadrant, 
with a diagonal matrix comprised of only the numbers $0$ and $-1$. 
Inserting this gives
%
%
\[\left(\Gamma_+(\overline{X}) s_\mu, \left(\prod_{k \geq n+1} (1-\psi_{k}\psi_{k}^*)\right) \sum_{i_k,j_k} Z'_{i_1,j_1}\cdots Z'_{i_c,j_c}
\psi_{i_1}\psi^*_{j_1} \cdots\psi_{i_c}\psi^*_{j_c} s_\nu\right)=\]
\[\left(\Gamma_+(\overline{X}) s_\mu, \rho(D_n) \exp\left(\rho'(Z')\right) s_\nu \right),\]
We can easily check that
\[D_n \exp(Z')= D_n(1+Z') = g_+^{-1}(X) D_n g_+(X).\]
Substituting back into \eqref{grassexp}, and using \eqref{gpm}, we are left with
\[\left(\Gamma_+(\overline{X})\varphi_{-p_1} \Gamma_+(\overline{X})s_\mu,
\Gamma_+^{-1}(X)\varphi_{-p_1} \pi_n \varphi_{-p_1}\Gamma_+(X) s_\nu\right),\]
which evaluates to the theorem, by some simple commutation relations.
\end{proof}

\section{Moduli spaces of sheaves}\label{sheaves} 
Before stating our main theorem, we recall the characters of the tangent spaces to a fixed point of $\cM_{r,n}$ under the torus action mentioned in the introduction,
so that we may give an explicit description of the inner product \eqref{modexp} using the localization formula.
This was worked out by Ellingsrud and G\"ottsche in \cite{EG}, and by Nakajima using the ADHM description \cite{Nak4}.

Let $\cM_{r,n}$ be the moduli space of framed rank-$r$ torsion-free sheaves $(F,\Phi)$ on $\Pt$ with second Chern class $c_2=n$,
referred to in the introduction. 
A framing $\Phi$ is a choice of an isomorphism
\[ \Phi : F\big|_{\Pp_\infty} \cong \cO_{\Pp_\infty}^{\oplus r}.\]
The two-dimensional torus $T$ acts on $\Pt$ by
\begin{equation}
\label{torus}
(z_1,z_2)\cdot (x_1,x_2,x_3) = (z_1^{-1}x_1,z_2^{-1}x_2,x_3), \quad (z_1,z_2) \in T,
\end{equation}
inducing an action on $\cM_{r,n}$ by pulling back sheaves.
There is a commuting action of the $r$-dimensional torus
\[T^r = \{(w_1,...,w_r)\} \subset GL_r,\]
by left-composition with the framing $\Phi$.
These combine to give an action of $T^2 \times T^r$ on $\cM_{r,n}$ with discrete fixed loci. 
When the rank is one, this moduli space is isomorphic to $\Hilb_n \C^2$, and the fixed points are the zero-dimensional subschemes whose ideals
are generated by monomials,
\[I_\mu = (x^{\mu_1},x^{\mu_2-1}y,...,y^{\ell(\mu)}) \subset R = \C[x,y],\]
indexed by partitions $\mu$ of size $|\mu|=n$. 
In higher rank, the fixed set is indexed by $r$-tuples of partitions $\tilde{\mu} = (\mu^1,...,\mu^r)$ with total degree $n$,
\[ F_{\tilde{\mu}} = I_{\mu^1} \oplus \cdots \oplus I_{\mu^r},\quad \sum_i |\mu^i| = n,\]
thought of as a sheaf on $\Pt$ which is trivial at the line at infinity. The torus $T^2$ acts diagonally on each component
by scaling each monomial $x^iy^j$ by $z_1^iz_2^j$, and $T^r$ acts by scaling by $w_i$ on the $i$th factor,
making $F_{\tilde{\mu}}$ an equivariant sheaf.

By deformation theory and the moduli functor, the tangent space to a sheaf $F$ is well-known to be
\begin{equation}
\label{TMrn}
T_F \cM_{r,n} = \Ext^1(F,F(-\Pp_\infty)).
\end{equation}
Using \c{C}ech cohomology, it turns out that one may compute \eqref{TMrn} using only the
chart $\C^2$. When the rank is one, the character of the tangent space at a fixed point $I_\mu$ is
\begin{equation}
\label{Tmu}
\ch T_\mu \Hilb_n \C^2 = E_{\mu,\mu},
\end{equation}
where
\[E_{\mu,\nu} = \chi(I_\emptyset,I_\emptyset) - \chi(I_\mu,I_\nu),\quad \chi(F,G) = \sum_{i = 0}^2 (-1)^i \ch \Ext^i_{\C^2} (F,G) ,\]
for torus-invariant sheaves $F,G$ on $\C^2$.
The above characters of the infinite-dimensional $\Ext$-groups are elements of $\C((z_1,z_2))$,
but their difference is a Laurent polynomial with nonnegative integer coefficients,
which is the character of a $2n$-dimensional representation of $G$. In general rank $r$, the tangent space \eqref{TMrn} splits
\[\ch T_{\tilde{\mu}} \cM_{r,n} = \chi(R^{\oplus r}, R^{\oplus r}) - \chi(F_{\tilde{\mu}},F_{\tilde{\mu}}) = \sum_{i,j} w_j w_i^{-1} E_{\mu^i,\mu^j},\]
by additivity of $\chi$.

Multiplicativity of $\chi$ produces combinatorial formulas for $E_{\mu,\nu}$,
\[\chi(I_\mu,I_\nu) = z_1^{-1}z_2^{-1} M\overline{V_\mu}V_\nu,\quad M = (1-z_1)(1-z_2),\]
\[V_\mu = \ch I_\mu = M^{-1}-U_\mu,\quad U_\mu = \sum_{\Box \in \mu} z_1^j z_2^i,\]
where a box is an element $(i,j)$ in the Young diagram of $\mu$.
Let us also set
\[V_{\tilde{\mu}} = \ch F_{\tilde{\mu}} = W M^{-1}-U_{\tilde{\mu}},\quad U_{\tilde{\mu}} = \sum_i w_i U_{\mu^i}, \quad W = \sum_i w_i.\]
Simplifying further produces the arm and leg length formula,
\begin{equation}
\label{armleg}
E_{\mu,\nu} = \sum_{\square\in \mu} z_1^{-a_\mu(\square)-1} \, z_2^{l_\nu(\square)}+
\sum_{\square\in \nu} z_1^{a_\nu(\square)}\, z_2^{-l_\mu(\square)-1}.
\end{equation}
where
\[a_\mu(\Box) = \mu_i-j, \quad l_\mu(\Box) = \mu'_j-i,\]
are the arm and leg lengths, which may be negative.
Equation \eqref{armleg} differs from \cite{Nak4} theorem 2.11 because we used the inverted action
\eqref{torus} so that the character of $I_\mu$ would be an expansion in $z_i$ rather than $z_i^{-1}$, and because
we have labeled our Young diagrams so that arm length corresponds to the $x$-coordinate.

By the localization formula combined with Riemann-Roch, the inner product discussed in the introduction is
\begin{equation}
\label{modexpdef}
(f,g)_{\cM_{r,n},k} = \sum_{|\tilde{\mu}| = n} f(\overline{U_{\tilde{\mu}}}) g(U_{\tilde{\mu}})\det(U_{\tilde{\mu}})^k \Omega(\ch T^*_{\tilde{\mu}} M_{r,n}) .
\end{equation}
%
We may now state our main theorem.
\begin{thm}
\label{modthm}
Fix $f,g \in \Lambda$, and $r,n \geq 0$.
Then $(f,g)_{\cM_{r,n},k}$ is the unique element of $\C(z_i,w_j)$ satisfying
\begin{enumerate}
\item[a.] There exists an element $F \in \C(z_i,w_j)[x_i,y_j]$ such that
\[(f,g)_{\cM_{r,n},k} = F(z_i^k,w_j^k)\]
for all $k$.
\item[b.] There exists $k_0$ such that for any $k \geq k_0$
\[(f,g)_{\cM_{r,n},k} = \Tr \varphi_{(1-M)p_1} m_f^* \Gamma_+(W)\pi_n m_{e_n^k} \Gamma_-(z_1z_2\overline{W}) m_g.\]
\end{enumerate}
\end{thm}
\begin{proof}
The first condition and the uniqueness statement are obvious.

Let $X \subset R^{\oplus r}$ be the subrepresentation of $T^2 \times T^r$ spanned by the monomials of degree at most $N$ in each variable,
whose character is
\[\ch X = \sum_{1 \leq a\leq r} w_a\sum_{1 \leq i,j \leq N} z_1^iz_2^j.\]
Given a torus-invariant subspace $V \subset X$, let
\[E(V) = E_0(V)+E_1(V)+E_2(V),\]
where
\[E_0(V) = \overline{V}(X-V),\quad 
E_1(V) = W z_1^{-1}z_2^{-1} \left(\overline{X}-\overline{V}\right),\]
\[ E_2(V) = (\overline{M}-1)\overline{V}(X-V).\]
Using
\begin{equation}
\label{limX}
\lim_{N \rightarrow \infty} X = WM^{-1},
\end{equation}
it is easy to see that
\begin{equation}
\lim_{N \rightarrow \infty} E(V_{\tilde{\mu}} \cap X) = \overline{M} V_{\emptyset^r} \overline{V}_{\emptyset^r} -\overline{M} V_{\tilde{\mu}} \overline{V}_{\tilde{\mu}} =
T_{\tilde{\mu}} M_{r,n}.
\end{equation}
If $V$ is not the intersection of $X$ with $V_{\tilde{\mu}}$, one can check that the constant term will always be negative,
which implies that $\Omega(E(V))$ is zero.

Now fix $z_i,w_j$ to be numbers with $|z_i| < 1$, but still consider $z_i,w_j,v$ as plethystic variables.
Also fix $f,g,n$, and let $k$ be a number large enough that the following function converges whenever $|v| < 1+\epsilon$ for some $\epsilon > 0$ depending on $|z_i|$:
\begin{equation}
\label{Fv}
F(v) = \lim_{N \rightarrow \infty} \sum_{V} f(\overline{V}')g(V') \det(V')^k
\Omega\left(\overline{E}_0+\overline{E}_1+v\overline{E}_2\right),
\end{equation}
over all $V \subset X$ of codimension $n$. 
By the preceding paragraph, we have
\begin{equation}
\label{F1}
F(1) = (f,g)_{\cM_{r,n},k},
\end{equation}
simply because only the only contributing terms are $V = X \cap V_{\tilde{\mu}}$, and the summands agree as $N$ approaches infinity. Since
\[ (f,g)_{\cM_{r,n},k} = (f, g e_n^k)_{\cM_{r,n},0},\]
we may assume that $g$ is a multiple of $e_n^k$, and set $k=0$ in the inner product.

Now notice that $E_0(V)$ is the character to the tangent space of the Grassmannian at $V$,
\[E_0(V) = \ch T_V \Gr_m X,\quad m = N^2-n.\]
By the Cauchy expansion
\begin{equation}
\label{cauchy}
\Omega\left(fXY\right) = \sum_{\mu} \frac{\varphi_f p_\mu}{\zee(\mu)} p_\mu(X)p_\mu(Y),
\end{equation}
applied to $\Omega(v\overline{E_2})$, the Taylor series of $F(v)$ about zero is given by
\begin{equation}
\label{nonconv}
\lim_{N\rightarrow \infty} \sum_{\mu} \frac{\varphi_{v(M-1)} p_\mu}{\zee(\mu)}
\left( f p_\mu, \left(\varphi_{X-p_1}\Omega(\overline{W} z_1z_2 p_1) g\right)p_\mu\right)_{\Gr_m},
\end{equation}
which converges in a small neighborhood of the origin.
The insertion of $\Omega(\overline{W}z_1z_2p_1)$ comes from $\Omega(\overline{E_1})$, and we have used
\[(\varphi_{X-p_1} g)(V) = g(V').\]
The series is valid only for small values of $v$, because in fact $F_N(v)$ has poles in the unit disc depending on $z_i,f,g,N$. 

For small $v$ we may switch the limits to get
\[\sum_{\mu} \frac{\varphi_{v(M-1)}p_\mu}{\zee(\mu)} \lim_{N \rightarrow \infty} (fp_\mu, \left(\varphi_{-p_1}\Gamma_+(X)\Gamma_-(\overline{W}z_1z_2) g\right) p_\mu)_{\Gr_m}=\]
\begin{equation}
\label{conv}
\sum_{\mu} \frac{\varphi_{v(M-1)}p_\mu}{\zee(\mu)}\left( fp_\mu, 
\Gamma_+(M^{-1}W)\pi_n\Gamma_-(\overline{W}z_1z_2) m_g\Gamma_+^{-1}(M^{-1}W) \varphi_{-p_1}p_\mu\right)
\end{equation}
by theorem \ref{grassthm}, equation \eqref{limX}, and the fact that $\Gamma_+$ is a homomorphism.
This sum actually satisfies the geometric series test for $v$ in some disc of radius $1+\epsilon$, because $M-1$ is first order in $z_i$, which outweighs the $v$ factor for small enough $\epsilon$,
and by some easy bounds on the remaining part of the summand. 

By analytic continuation, equation \eqref{conv} agrees with $F(v)$ for all $v$ in the larger disc.
We recover the theorem by setting $v=1$, and rewriting \eqref{conv} as a trace,
\[\Tr \varphi_{(M-1)p_1} m_f^* 
\Gamma_+(M^{-1}W)\pi_n\Gamma_-(\overline{W}z_1z_2) m_g\Gamma_+^{-1}(M^{-1}W) \varphi_{-p_1}=\]
\[\Tr \varphi_{(1-M)p_1} \Gamma_+((M-1)M^{-1}W) m_f^* 
\Gamma_+(M^{-1}W)\pi_n\Gamma_-(\overline{W}z_1z_2) m_g = \]
\[\Tr \varphi_{(1-M)p_1} m_f^* 
\Gamma_+(W)\pi_n\Gamma_-(\overline{W}z_1z_2) m_g.\]
\end{proof}
\section{Examples}
We conclude the paper with two simple examples, leaving the applications mentioned in the introduction for future papers.

\emph{Example.} If $r=0$, theorem \ref{modthm} says that the inner product is uniquely determined by
\[(f,g)_{\cM_{0,n},k} = Tr \varphi_{(1-M)p_1} m_f^* \pi_n m_{e_n^k} m_g,\]
for large $k$. The right-hand-side is not always zero, but it is for large $k$, for degree reasons.
we have therefore recovered the vacuous statement that
the left-hand-side is zero for all $k$, being a sum with no terms.

\emph{Example.} When the rank is one, we lose nothing by setting $w_1=1$. Since $U_\mu$ always has a constant term of of one,
we find that
\[\Omega(-U_\mu) = \Omega(-\overline{U_\mu}) = 0,\]
and therefore, the inner-product must vanish if $f$ or $g$ is a multiple of $\Omega(-p_1)$. The first case is clear from the theorem,
\[(f\Omega(-p_1),g)_{\Hilb_n,k} = Tr \varphi_{(1-M)p_1} m_f^* \pi_n m_{e_n^k} \Gamma_-(z_1z_2) m_g\]
for large $k$. Again, this vanishes for large $k$ for degree reasons, since we have canceled the first vertex operator.
The case when $g$ is a multiple of $\Omega(-p_1)$ does not appear to be transparent.

\emph{Example.} When $f=g=1$, the inner product is the massless $K$-theoretic partition function $Z_n$ studied by Nakajima in \cite{Nak5}.
It turns out that $k_0 = 0$ is sufficiently large for conditional convergence. The partition function stabilizes as for large $n$ to
\[Z_\infty = \lim_{n\rightarrow \infty} Z_n = \Tr \varphi_{(1-M)p_1}\Gamma_+(W) \Gamma_-(\overline{W}z_1z_2),\]
by the theorem.
For any $f,g \in \Lambda$, and $h$ a rational function in some variables, it is easy to show that
\[\Tr \varphi_{(1-h^{-1})p_1} m^*_f m_g = \sum_{\mu} \frac{\varphi_{(1-h^{-1})}(p_\mu)}{\zee(\mu)} (f p_\mu,g p_\mu) =C(f,g)_{h},\]
where
\[\quad C = \sum_{\mu} \varphi_{(1-h^{-1})} (p_\mu),\]
and the inner product is defined by
\[(p_\mu,p_\nu)_{h} = \delta_{\mu,\nu}\zee(\mu) \varphi_{h} (p_\mu).\]
In particular, $(f,g)_{M^{-1}}$ is a different normalization of the Macdonald inner product which appears in Haiman's theory, which we
would expect to appear in rank one.

The final answer is
\[C^{-1} Z_\infty = (\Gamma_-(W)1,\Gamma_-(z_1z_2 \overline{W})1)_{M^{-1}} = \]
\[(1,\Gamma_+(WM^{-1})\Gamma_-(\overline{W}z_1z_2)1)_{M^{-1}} = \Omega(W \overline{W}M^{-1}z_1z_2),\]
where
\[C = \sum_{\mu} \varphi_{(1-M)} (p_\mu).\]

\end{document}